\documentclass[12pt]{article}
\usepackage{amsfonts}
\usepackage{amssymb}
\usepackage{amsmath}

\newtheorem{theorem}{Theorem}[section]

\newtheorem{corollary}{Corollary}[section]

\newtheorem{example}{Example}[section]

\newtheorem{lemma}{Lemma}[section]

\newenvironment{proof}[1][Proof]{\noindent\textbf{#1.} }{\ \rule{0.5em}{0.5em}}
\input{tcilatex}
\begin{document}

\title{On Fillmore's theorem extended by Borobia\thanks{%
Supported by Fondecyt 1170313, Chile; Conicyt-PAI 79160002, 2016, Chile.}}
\author{Ana I. Julio, Ricardo L. Soto\thanks{%
E-mail addresses: rsoto@ucn.cl (R.L. Soto), ajulio@ucn.cl (A.I. Julio)} \\
{\small Dpto. Matem\'{a}ticas, Universidad Cat\'{o}lica del Norte, Casilla
1280}\\
{\small Antofagasta, Chile.}}
\date{}
\maketitle

\begin{abstract}
Fillmore Theorem says that if $A$ is an $n\times n$ complex non-scalar
matrix and $\gamma _{1},\ldots ,\gamma _{n}$ are complex numbers with $%
\gamma _{1}+\cdots +\gamma _{n}=trA,$ then there exists a matrix $B$ similar
to $A$ with diagonal entries $\gamma _{1},\ldots ,\gamma _{n}.$ Borobia
simplifies this result and extends it to matrices with integer entries.
Fillmore and Borobia do not consider the nonnegativity hypothesis. Here, we
introduce a different and very simple way to compute the matrix $B$ similar
to $A$ with diagonal $\gamma _{1},\ldots ,\gamma _{n}.$ Moreover, we
consider the nonnegativity hypothesis and we show that for a list $\Lambda
=\{\lambda _{1},\ldots ,\lambda _{n}\}$ of complex numbers of Suleimanova or 
\v{S}migoc type, and a given list $\Gamma =\{\gamma _{1},\ldots ,\gamma
_{n}\}$ of nonnegative real numbers, the remarkably simple condition $\gamma
_{1}+\cdots +\gamma _{n}=\lambda _{1}+\cdots +\lambda _{n}$ is necessary and
sufficient for the existence of a nonnegative matrix with spectrum $\Lambda $
and diagonal entries $\Gamma .$ This surprising simple result improves a
condition recently given by Ellard and \v{S}migoc in arXiv:.1702.02650v1.
\end{abstract}

\textit{AMS classification: \ \ 15A18, 15A29}

\textit{Key words: Inverse eigenvalue problem, nonnegative matrix, diagonal
entries.}

\section{Introduction}

\noindent The problem of the existence and construction of nonnegative
matrices with prescribed eigenvalues and diagonal entries is an important
inverse problem, interesting by itself but also necessary to apply a
perturbation result, due to R. Rado and published by H. Perfect \cite%
{Perfect1}, which has played an important role in the study of the
nonnegative inverse eigenvalue problem and the nonnegative inverse
elementary divisors problem. Fillmore \cite{Fillmore} proved that if $A$ is
an $n\times n$ non-scalar matrix over a field $\mathbb{F},$ and $\gamma
_{1},\ldots ,\gamma _{n}\in \mathbb{F}$ with $\sum\limits_{i=1}^{n}\gamma
_{i}=trA,$ then there exists a matrix $B$ similar to $A$ having diagonal
entries $\gamma _{1},\ldots ,\gamma _{n}.$ Borobia \cite{Borobia} develops
an explicit and simple algorithm to compute the matrix $B,$ and then extends
Fillmore Theorem to matrices with integer entries. Results of Fillmore and
Borobia do not consider the nonnegativity hypothesis. Here, by applying a
well known result of Brauer \cite{Brauer}, we present a different and very
simple way to compute the matrix $B$ similar to $A$ with arbitrarily
prescribed diagonal entries $\gamma _{1},\ldots ,\gamma _{n}$ (except for
the trace property). In this work we also consider the nonnegativity
hypothesis and we prove that given a realizable list of complex numbers $%
\Lambda =\{\lambda _{1},\lambda _{2},\ldots ,\lambda _{n}\},$ $\lambda _{1}$
being the Perron eigenvalue, with%
\begin{equation}
i)\ \mathcal{F}=\{\lambda _{i}\in \mathbb{C}:Re\lambda _{i}\leq 0,\text{ }%
\left\vert Re\lambda _{i}\right\vert \geq \left\vert Im\lambda
_{i}\right\vert \ i=2,\ldots ,n\},  \label{Co1}
\end{equation}%
or%
\begin{equation}
ii)\ \mathcal{G}=\{\lambda _{i}\in \mathbb{C}:Re\lambda _{i}\leq 0,\text{ }%
\left\vert \sqrt{3}Re\lambda _{i}\right\vert \geq \left\vert Im\lambda
_{i}\right\vert ,\text{ }i=2,\ldots ,n,\}  \label{Co2}
\end{equation}%
and given real nonnegative numbers $\gamma _{1},\ldots ,\gamma _{n},$ such
that 
\begin{equation}
\sum\limits_{i=1}^{n}\gamma _{i}=\sum\limits_{i=1}^{n}\lambda _{i},
\label{Co3}
\end{equation}%
then there exists an $n\times n$ nonnegative matrix $B$ with spectrum $%
\Lambda $ and diagonal entries $\gamma _{1},\ldots ,\gamma _{n}$ if and only
if (\ref{Co3}) holds. It is well known that under condition (\ref{Co3}),
lists with the property $i)$ (lists of Suleimanova type) and list with the
property $ii)$ (lists of \v{S}migoc type) are realizable. The novelty here
is that, under the remarkably simple condition (\ref{Co3}), $\Lambda $ is
not only realizable, but it is realizable with arbitrary prescribed diagonal
entries. This surprising simple result improves, for lists $i)$ and $ii),$
the necessary and sufficient condition given by Ellard and \v{S}migoc in 
\cite{Ellard}.

\bigskip

\noindent The set of all matrices with constant row sums equal to $\alpha $
will be denoted by $\mathcal{CS}_{\alpha }.$ It is clear that $\mathbf{e}%
=[1,1,\ldots ,1]^{T}$ is an eigenvector of any matrix $A\in \mathcal{CS}%
_{\alpha },$ corresponding to the eigenvalue $\alpha .$ The importance of
matrices with constant row sums is due to the well known fact that if $%
\lambda _{1}$ is the desired Perron eigenvalue, then the problem of finding
a nonnegative matrix with spectrum $\Lambda =\{\lambda _{1},\lambda
_{2},\ldots ,\lambda _{n}\}$ is equivalent to the problem of finding a
nonnegative matrix in $\mathcal{CS}_{\lambda _{1}}$ with spectrum $\Lambda .$
We denote by $E_{i,j}$ the matrix with one in position $(i,j)$ and zeros
elsewhere. We shall say that the list $\Lambda =\{\lambda _{1},\lambda
_{2},\ldots ,\lambda _{n}\}$ is realizable if there is a nonnegative matrix $%
A$ with spectrum $\Lambda .$ In this case we say that $A$ is the realizing
matrix. Our main tools will be:

\begin{theorem}
\cite{Brauer}\label{Bra}Let $A$ be an $n\times n$ arbitrary matrix with
eigenvalues\newline
$\lambda _{1},\lambda _{2},\ldots ,\lambda _{n}.$ Let $\mathbf{v}=\left[
v_{1},v_{2},\ldots ,v_{n}\right] ^{T}$ an eigenvector of $A$ associated with
the eigenvalue $\lambda _{k}$ and let $\mathbf{q}$ be any $n-$dimensional
vector. Then the matrix $A+\mathbf{vq}^{T}$ has eigenvalues $\lambda
_{1},\ldots ,\lambda _{k-1},\lambda _{k}+\mathbf{v}^{T}\mathbf{q},\lambda
_{k+1},\ldots ,\lambda _{n},$ $k=1,2,\ldots ,n.$
\end{theorem}

\begin{lemma}
\cite{Soto3}\label{SoCa}Let $A\in \mathcal{CS}_{\lambda _{1}}$ with Jordan
canonical form $J(A).$ Let $\mathbf{q}=\left[ q_{1},\ldots ,q_{n}\right]
^{T} $ an arbitrary $n-$dimensional vector such that $\lambda
_{1}+\sum\limits_{i=1}^{n}q_{i}\neq \lambda _{i},$ $i=2,\ldots ,n.$ Then the
matrix $A+\mathbf{eq}^{T}$ has Jordan form $J(A)+\left(
\sum\limits_{i=1}^{n}q_{i}\right) E_{11}.$ In particular, if $%
\sum\limits_{i=1}^{n}q_{i}=0,$ then $A$ and $A+\mathbf{eq}^{T}$ are similar.
\end{lemma}

\begin{lemma}
\label{perfect}\cite{Soto25} The complex numbers $\lambda _{1},\lambda
_{2},\lambda _{3}$, $\lambda _{1}\geq |\lambda _{i}|$, $i=2,3$ and $\gamma
_{1},\gamma _{2},\gamma _{3}$ are, respectively, the eigenvalues and the
diagonal entries of a $3\times 3$ nonnegative matrix $B\in \mathcal{CS}%
_{\lambda _{1}}$ if only if 
\begin{align*}
& i)\ 0\leq \gamma _{k}\leq \lambda _{1},\ k=1,2,3 \\
& ii)\ \gamma _{1}+\gamma _{2}+\gamma _{3}=\lambda _{1}+\lambda _{2}+\lambda
_{3} \\
& iii)\ \gamma _{1}\gamma _{2}+\gamma _{1}\gamma _{3}+\gamma _{2}\gamma
_{3}\geq \lambda _{1}\lambda _{2}+\lambda _{1}\lambda _{3}+\lambda
_{2}\lambda _{3} \\
& iv)\ \max_{k}\gamma _{k}\geq Re\lambda _{2}
\end{align*}
\end{lemma}

\begin{lemma}
\label{Smigoc}\cite{smigoc} Let $A=%
\begin{bmatrix}
A_{1} & \mathbf{a} \\ 
\mathbf{b}^{T} & c%
\end{bmatrix}%
$ be an $n\times n$ matrix and let $B$ be an $m\times m$ matrix with Jordan
canonical form $J(B)=%
\begin{bmatrix}
c & \mathbf{0} \\ 
\mathbf{0} & I(B)%
\end{bmatrix}%
$. Then the matrix 
\begin{equation*}
C=%
\begin{bmatrix}
A_{1} & \mathbf{at}^{T} \\ 
\mathbf{sb}^{T} & B%
\end{bmatrix}%
,B\mathbf{s}=c\mathbf{s},\ \mathbf{t}^{T}B=c\mathbf{t}^{T},\text{with}\ 
\mathbf{t}^{T}\mathbf{s}=1
\end{equation*}%
has Jordan form 
\begin{equation*}
J(C)=%
\begin{bmatrix}
J(A) & 0 \\ 
0 & I(B)%
\end{bmatrix}%
.
\end{equation*}
\end{lemma}

\section{The general case}

\noindent In this section, by applying Theorem \ref{Bra}, we introduce a
very simple way to compute a matrix $B$ similar to a given non-sacalar
matrix $A,$ with $B$ having diagonal entries $\gamma _{1},\ldots
,\gamma_{n}. $ Of course, we need that $\sum\limits_{i=1}^{n}\gamma
_{i}=trA. $ Then we have:

\begin{theorem}
Let $A=\left( a_{ij}\right) $ be an $n\times n$ non-scalar complex matrix
and let $\Gamma =\{\gamma _{1},\ldots ,\gamma _{n}\}$ be a given list of
complex numbers such that $\sum\limits_{i=1}^{n}\gamma _{i}=trA.$ Let $%
\mathbf{x}=\left[ x_{1},\ldots ,x_{n}\right] $ be an eigenvector of $A$
having all its entries nonzero. Then there exists a matrix $B$ similar to $A$
with diagonal entries $\gamma _{1},\ldots ,\gamma _{n}.$
\end{theorem}

\begin{proof}
First, let $A$ be non-diagonal and let $D=diag\{x_{1},\ldots ,x_{n}\}.$ Then 
$D$ is a nonsingular matrix and $D^{-1}AD\in \mathcal{CS}_{\lambda },$ where 
$A\mathbf{x}=\lambda \mathbf{x.}$ Let 
\begin{equation*}
\mathbf{q}^{T}=\left[ q_{1},\ldots ,q_{n}\right] \text{ with }q_{i}=\gamma
_{i}-a_{ii},\text{ }i=1,\ldots ,n.
\end{equation*}%
Then $B=D^{-1}AD+\mathbf{eq}^{T}$ has diagonal entries $\gamma _{1},\ldots
,\gamma _{n}$ and from Lemma \ref{SoCa}, $B$ is similar to $A.$\newline
Second, let $A$ be diagonal. Consider 
\begin{equation*}
S=\left[ 
\begin{array}{cc}
1 & \mathbf{0}^{T} \\ 
\mathbf{e} & I_{n-1}%
\end{array}%
\right] \text{ \ and \ }S^{-1}=\left[ 
\begin{array}{cc}
1 & \mathbf{0}^{T} \\ 
-\mathbf{e} & I_{n-1}%
\end{array}%
\right] .
\end{equation*}%
Then $S^{-1}AS\in \mathcal{CS}_{\lambda },$ where $\lambda $ is an
eigenvalue of $A$ (a diagonal entry of $A$), and 
\begin{equation*}
B=S^{-1}AS+\mathbf{eq}^{T},\text{ \ with \ }q_{i}=\gamma _{i}-a_{ii},
\end{equation*}%
is similar to $A$ and has diagonal entries $\gamma _{1},\ldots ,\gamma _{n}.$
\end{proof}

\noindent Let us consider the Example $6$ in \cite{Borobia}:

\begin{example}
We want to construct a matrix with diagonal entries $3,5,-2,6,-1$ and
similar to%
\begin{equation*}
A=\left[ 
\begin{array}{ccccc}
4 & 0 & 4 & -3 & 5 \\ 
2 & 3 & 0 & 2 & 3 \\ 
0 & -2 & 2 & 5 & 4 \\ 
7 & 1 & 3 & 4 & 0 \\ 
2 & 5 & 3 & 0 & -2%
\end{array}%
\right] .
\end{equation*}%
$A$ has not constant row sums, but it has an eigenvector with all nonzero
entries $\mathbf{x}^{T}\mathbf{=}\left[ x_{i}\right] _{i=1}^{5}=\left[ 
\begin{array}{ccccc}
\frac{497}{601} & \frac{1259}{1033} & \frac{335}{241} & \frac{1698}{899} & 1%
\end{array}%
\right] .$ Let $D=diag\{x_{1},\ldots ,x_{5}\}.$ Then%
\begin{equation*}
D^{-1}AD=\left[ 
\begin{array}{ccccc}
4 & 0 & \frac{3893}{579} & -\frac{1713}{250} & \frac{2352}{389} \\ 
\frac{802}{591} & 3 & 0 & \frac{1652}{533} & \frac{3131}{1272} \\ 
0 & -\frac{2078}{1185} & 2 & \frac{2901}{427} & \frac{964}{335} \\ 
\frac{2081}{679} & \frac{533}{826} & \frac{1742}{789} & 4 & 0 \\ 
\frac{994}{601} & \frac{1298}{213} & \frac{1005}{241} & 0 & -2%
\end{array}%
\right] \in \mathcal{CS}_{\frac{5197}{524}}.
\end{equation*}%
Now we apply Theorem \ref{Bra} to obtain%
\begin{equation*}
B=D^{-1}AD+\mathbf{e}\left[ -1,2,-4,2,1\right]
\end{equation*}%
with the required diagonal entries $3,5,-2,6,-1.$ Moreover, from Lemma \ref%
{SoCa} $B$ is similar to $A.$
\end{example}

\subsection{The nonnegative case}

\noindent In this section we show that if (\ref{Co3}) is satisfied, then
lists $\Lambda =\{\lambda _{1},\lambda _{2},\ldots ,\lambda _{n}\}$ with
only one positive eigenvalue, of type $i)$ in (\ref{Co1}) and of type $ii)$
in (\ref{Co2}), are always realizable by a nonnegative matrix $B$ with
arbitrarily prescribed diagonal entries $\gamma _{1},\gamma _{2},\ldots
,\gamma _{n}$ (except for the condition (\ref{Co3})). Although lists in (\ref%
{Co2}) contain lists in (\ref{Co1}), we shall prove both, since the proofs
are different.

\begin{theorem}
\label{Tnon1}Let $\Lambda =\{\lambda _{1},\lambda _{2},\ldots ,\lambda _{n}\}
$ be a realizable list of complex numbers with $\lambda _{i}\in \mathcal{F},$
$i=2,\ldots ,n.$ Let $\Gamma =\{\gamma _{1},\ldots ,\gamma _{n}\}$ be a list
of nonnegative real numbers such that $\sum\limits_{i=1}^{n}\gamma
_{i}=\sum\limits_{i=1}^{n}\lambda _{i}.$ Then there exists a nonnegative
matrix $B\in \mathcal{CS}_{\lambda _{1}}$ with spectrum $\Lambda $ and
diagonal entries $\Gamma .$ If $A\in \mathcal{CS}_{\lambda _{1}}$ has the
spectrum $\Lambda ,$ then $B\in \mathcal{CS}_{\lambda _{1}}$ is similar to $%
A.$
\end{theorem}

\begin{proof}
Let $\lambda _{1},\lambda _{2},\ldots ,\lambda _{p}$ be real numbers
(nonpositive numbers, of course) and let $\lambda
_{p+1}=x_{p+1}+iy_{p+1},\ldots ,\lambda _{n}=x_{n-1}-iy_{n-1}$ be complex
nonreal numbers. Then the matrix%
\begin{equation*}
A=\left[ 
\begin{array}{ccccccccc}
\lambda _{1} &  &  &  &  &  &  &  &  \\ 
\lambda _{1}-\lambda _{2} & \lambda _{2} &  &  &  &  &  &  &  \\ 
\vdots &  & \ddots &  &  &  &  &  &  \\ 
\lambda _{1}-\lambda _{p} &  &  & \lambda _{p} &  &  &  &  &  \\ 
\lambda _{1}-x_{p+1}+y_{p+1} &  &  &  & x_{p+1} & -y_{p+1} &  &  &  \\ 
\lambda _{1}-x_{p+1}-y_{p+1} &  &  &  & y_{p+1} & x_{p+1} &  &  &  \\ 
\vdots &  &  &  &  &  & \ddots &  &  \\ 
\lambda _{1}-x_{n-1}+y_{n-1} &  &  &  &  &  &  & x_{n-1} & -y_{n-1} \\ 
\lambda _{1}-x_{n-1}-y_{n-1} &  &  &  &  &  &  & y_{n-1} & x_{n-1}%
\end{array}%
\right]
\end{equation*}%
has spectrum $\Lambda $ and constant row sums equal to $\lambda _{1}.$ Let $%
\mathbf{q}^{T}=\left[ q_{1},q_{2},\ldots ,q_{n}\right] $ with $q_{i}=\gamma
_{i}-Re\lambda _{i},$ $i=1,\ldots ,n.$ then $B=A+\mathbf{eq}^{T}$ is
nonnegative with diagonal entries $\gamma _{1},\ldots ,\gamma _{n}$ and from
Lemma \ref{SoCa} it is also similar to $A.$
\end{proof}

\begin{corollary}
Let $\Lambda =\{\lambda _{1},\lambda _{2},\ldots ,\lambda _{n}\}$ be a
realizable list of real numbers with $\lambda _{i}\leq 0,$ $i=2,\ldots ,n.$
Let $\Gamma =\{\gamma _{1},\ldots ,\gamma _{n}\}$ be a list of nonnegative
real numbers such that $\sum\limits_{i=1}^{n}\gamma
_{i}=\sum\limits_{i=1}^{n}\lambda _{i}.$ Then there exists a nonnegative
matrix $B\in \mathcal{CS}_{\lambda _{1}}$ with spectrum $\Lambda $ and
diagonal entries $\Gamma .$
\end{corollary}

\noindent Observe that if $\sum\limits_{i=1}^{n}\lambda _{i}>0,$ then in
Theorem \ref{Tnon1} we may compute a positive matrix $B$ with diagonal
entries $\gamma _{1},\ldots ,\gamma _{n}.$

\bigskip

\noindent Next, we show that given a list of complex numbers of \v{S}migoc
type $\Lambda =\{\lambda _{1},\ldots ,\lambda _{n}\}$ and a list of
nonnegative real numbers $\Gamma =\{\gamma _{1},\ldots ,\gamma _{n}\}$ such
that $\sum\limits_{i=1}^{n}\gamma _{i}=\sum\limits_{i=1}^{n}\lambda _{i},$
then there exists a nonnegative matrix with spectrum $\Lambda $ and diagonal
entries $\gamma _{1},\ldots ,\gamma _{n}.$ We shall need the following
lemmas:

\begin{lemma}
\label{caso 3} Let $\Lambda =\{\lambda _{1},x+iy,x-iy\}$ be a realizable
list of \v{S}migoc type, and let $\Gamma =\{\gamma _{1},\gamma _{2},\gamma
_{3}\}$ be a list of nonnegative real numbers with $\gamma _{1}+\gamma
_{2}+\gamma _{3}=\lambda _{1}+2x$. Then there exists a nonnegative matrix $%
B\in \mathcal{CS}_{\lambda _{1}}$ with spectrum $\Lambda $ and diagonal
entries $\gamma _{1},\gamma _{2},\gamma _{3}.$
\end{lemma}

\begin{proof}
Since $\Lambda $ is of \v{S}migoc type, $x\leq 0,-\sqrt{3}x\geq y.$ Then it
is easy to check that $\Lambda $ satisfies conditions $i)$ $ii)$ $iii)$ $iv)$
from Lemma \ref{perfect} and a matrix 
\begin{equation*}
B=%
\begin{bmatrix}
\gamma _{1} & 0 & \lambda _{1}-\gamma _{1} \\ 
\lambda _{1}-\gamma _{2}-p & \gamma _{2} & p \\ 
0 & \lambda _{1}-\gamma _{3} & \gamma _{3}%
\end{bmatrix}%
,
\end{equation*}%
where $p=\frac{1}{\lambda _{1}-\gamma _{3}}[\gamma _{1}\gamma _{2}+\gamma
_{1}\gamma _{3}+\gamma _{2}\gamma _{3}-(2\lambda _{1}x+x^{2}+y^{2})]$, is
nonnegative with spectrum $\Lambda $ and with diagonal entries $\gamma
_{1},\gamma _{2},\gamma _{3}$.
\end{proof}

\begin{lemma}
\label{L smigoc} Let $\Lambda =\{\lambda _{1},\lambda _{2},\ldots ,\lambda
_{n}\}$ be a realizable list of complex numbers, $n\geq 3$, with $\lambda
_{i}\in \left( \mathcal{G}-\mathcal{F}\right) $, $i=2,\ldots ,n.$ Let $%
\Gamma =\{\gamma _{1},\cdots ,\gamma _{n}\}$ be a list of nonnegative real
numbers such that $\sum_{i=1}^{n}\gamma _{i}=\sum_{i=1}^{n}\lambda _{i}.$
Then there exists a nonnegative matrix $B$ with spectrum $\Lambda $ and
diagonal entries $\gamma _{1},\ldots ,\gamma _{n}$.
\end{lemma}

\begin{proof}
If $n=3$, the result follows from Lemma \ref{caso 3}. Suppose the result is
true for lists with $n-2$ numbers. We take the partition $\Lambda =\Lambda
_{1}^{\prime }\cup \Lambda _{2}^{\prime },$ where%
\begin{equation*}
\Lambda _{1}^{\prime }=\{\lambda _{1},\lambda _{2},\ldots ,\lambda _{n-2}\},%
\text{ \ }\Lambda _{2}^{\prime }=\{\lambda _{n-1},\lambda _{n}\}
\end{equation*}%
with associated lists%
\begin{equation*}
\Lambda _{1}=\Lambda _{1}^{\prime },\text{ \ }\Lambda _{2}=\{c,\lambda
_{n-1},\lambda _{n}\},\text{ \ }c=\sum_{i=1}^{n-2}\lambda
_{i}-\sum_{i=1}^{n-3}\gamma _{i}.
\end{equation*}%
From the induction hypothesis, $\Lambda _{1}$ is realizable by a nonnegative
matrix 
\begin{equation*}
A_{1}=%
\begin{bmatrix}
A_{11} & \mathbf{a} \\ 
\mathbf{b}^{T} & c%
\end{bmatrix}%
\end{equation*}

\noindent with diagonal entries $\{\gamma _{1},\gamma _{2},\cdots ,\gamma
_{n-3},c\}.$ From Lemma \ref{caso 3}, $\Lambda _{2}$ is realizable by a
nonnegative matrix 
\begin{equation*}
A_{2}=%
\begin{bmatrix}
\gamma _{n-2} & 0 & c-\gamma _{n-2} \\ 
c-\gamma _{n-1}-p & \gamma _{n-1} & p \\ 
0 & c-\gamma _{n} & \gamma _{n}%
\end{bmatrix}%
\in \mathcal{CS}_{c},
\end{equation*}%
with diagonal entries $\gamma _{n-2},\gamma _{n-1},\gamma _{n}.$ Finally,
from Lemma \ref{Smigoc} with $A_{2}\mathbf{e}=c\mathbf{e}$, $\mathbf{t}%
^{T}A_{2}=c\mathbf{t}^{T}$, and $\mathbf{t}^{T}\mathbf{e}=1,$ a matrix 
\begin{equation*}
B=%
\begin{bmatrix}
A_{11} & \mathbf{at}^{T} \\ 
\mathbf{eb}^{T} & A_{2}%
\end{bmatrix}%
\end{equation*}%
is nonnegative with spectrum $\Lambda $ and diagonal entries $\gamma
_{1},\ldots ,\gamma _{n}$.
\end{proof}

\begin{theorem}
Let $\Lambda =\{\lambda _{1},\lambda _{2},\ldots ,\lambda _{n}\},$ $n\geq 3$%
, be a realizable list of complex numbers, $\lambda _{i}\in \mathcal{G}$, $%
i=2,\ldots ,n$. Let $\Gamma =\{\gamma _{1},\cdots ,\gamma _{n}\}$ be a list
of nonnegative real numbers such that $\sum_{i=1}^{n}\gamma
_{i}=\sum_{i=1}^{n}\lambda _{i}$. Then there exists a nonnegative matrix $B$
with spectrum $\Lambda $ and diagonal entries $\gamma _{1},\ldots ,\gamma
_{n}.$
\end{theorem}

\ 

\begin{proof}
Let $\{\lambda _{2},\ldots ,\lambda _{p}\}\subset \mathcal{F}$ and let $%
\{\lambda _{p+1},\ldots ,\lambda _{n}\}\subset \left( \mathcal{G-F}\right) $%
. We take the partition $\Lambda =\Lambda _{1}^{\prime }\cup \Lambda
_{2}^{\prime },$ where 
\begin{equation*}
\Lambda _{1}^{\prime }=\{\lambda _{1},\lambda _{2},\ldots ,\lambda _{p}\},%
\text{ \ }\Lambda _{2}^{\prime }=\{\lambda _{p+1},\ldots ,\lambda _{n}\}
\end{equation*}%
with associated lists%
\begin{equation*}
\Lambda _{1}=\Lambda _{1}^{\prime },\text{ \ \ }\Lambda _{2}=\{c,\lambda
_{p+1},\ldots ,\lambda _{n}\},\ \ c=\sum_{i=1}^{p}\lambda
_{i}-\sum_{i=1}^{p-1}\gamma _{i}.
\end{equation*}%
From Theorem \ref{Tnon1}, $\Lambda _{1}$ is realizable by a nonnegative
matrix 
\begin{equation*}
A_{1}=%
\begin{bmatrix}
A_{11} & \mathbf{a} \\ 
\mathbf{b}^{T} & c%
\end{bmatrix}%
\in \mathcal{CS}_{\lambda _{1}}
\end{equation*}%
with diagonal entries $\gamma _{1},\ldots ,\gamma _{p-1},c.$ Moreover, from
Lemma \ref{L smigoc} $\Lambda _{2}$ is realizable by a nonnegative matrix $%
A_{2}$ with diagonal entries $\gamma _{p},\ldots ,\gamma _{n}$. Finally,
from Lemma \ref{Smigoc} a matrix 
\begin{equation*}
B=%
\begin{bmatrix}
A_{11} & \mathbf{at}^{T} \\ 
\mathbf{eb}^{T} & A_{2}%
\end{bmatrix}%
,
\end{equation*}%
with $A_{2}\mathbf{e}=c\mathbf{e}$, $\mathbf{t}^{T}A_{2}=c\mathbf{t}^{T}$,
and $\mathbf{t}^{T}\mathbf{e}=1$, is nonnegative with spectrum $\Lambda $
and diagonal entries $\gamma _{1},\ldots ,\gamma _{n}.$
\end{proof}

\begin{example}
Consider $\Lambda =\{16,-1,-2,-2+2i,-2-2i,-2+3i,-2-3i\}$. We want compute a
nonnegative matrix with spectrum $\Lambda $ and diagonal entries $%
\{0,1,2,0,2,0,0\}$. \newline
Then we consider the lists: 
\begin{equation*}
\Lambda _{1}=\{16,-1,-2,-2+2i,-2-2i\},\ \ \Lambda _{2}=\{c,-2+3i,-2-3i\}.
\end{equation*}%
From Theorem \ref{Tnon1} we compute a nonnegative matrix $A_{1}$ with
spectrum $\Lambda _{1}$ and diagonal entries $0,1,2,0,c,$ where $c$ must be
equal to $6$: \newline
\begin{equation*}
A_{1}=%
\begin{bmatrix}
16 &  &  &  &  \\ 
17 & -1 &  &  &  \\ 
18 & 0 & -2 &  &  \\ 
20 & 0 & 0 & -2 & -2 \\ 
16 & 0 & 0 & 2 & -2%
\end{bmatrix}%
+\mathbf{e}(-16,2,4,2,8)^{T}=%
\begin{bmatrix}
0 & 2 & 4 & 2 & 8 \\ 
1 & 1 & 4 & 2 & 8 \\ 
2 & 2 & 2 & 2 & 8 \\ 
4 & 2 & 4 & 0 & 6 \\ 
0 & 2 & 4 & 4 & 6%
\end{bmatrix}%
.
\end{equation*}%
From Lemma \ref{caso 3} we compute a nonnegative matrix $A_{2}$ with
spectrum $\Lambda _{2}=\{c=6,-2+3i,-2-3i\}$ and digoanal entries $\{2,0,0\}$%
: 
\begin{equation*}
A_{2}=%
\begin{bmatrix}
2 & 0 & 4 \\ 
\frac{25}{6} & 0 & \frac{11}{6} \\ 
0 & 6 & 0%
\end{bmatrix}%
.
\end{equation*}%
Finally from Lemma \ref{Smigoc} with 
\begin{equation*}
\mathbf{a}=[8,8,8,6]^{T},\ \mathbf{b}^{T}=[0,2,4,4],
\end{equation*}%
and 
\begin{equation*}
\mathbf{s}=\mathbf{e},\mathbf{t}^{T}=\bigg[\frac{25}{73},\frac{24}{73},\frac{%
24}{73}\bigg],
\end{equation*}%
we have that 
\begin{equation*}
B=%
\begin{bmatrix}
0 & 2 & 4 & 2 & \frac{200}{73} & \frac{192}{73} & \frac{192}{73} \\ 
1 & 1 & 4 & 2 & \frac{200}{73} & \frac{192}{73} & \frac{192}{73} \\ 
2 & 2 & 2 & 2 & \frac{200}{73} & \frac{192}{73} & \frac{192}{73} \\ 
4 & 2 & 4 & 0 & \frac{150}{73} & \frac{144}{73} & \frac{144}{73} \\ 
0 & 2 & 4 & 4 & 2 & 0 & 4 \\ 
0 & 2 & 4 & 4 & \frac{25}{6} & 0 & \frac{11}{6} \\ 
0 & 2 & 4 & 4 & 0 & 6 & 0%
\end{bmatrix}%
\end{equation*}%
is nonnegative with spectrum $\Lambda $ and diagonal entries $%
\{0,1,2,0,2,0,0\}.$ Observe that we may obtain the diagonal entries in any
desired order.
\end{example}

\end{document}